\documentclass[12pt]{amsart}
\usepackage{amsmath,amsfonts,amssymb,amsxtra,amscd,enumerate,amsthm}
\usepackage{color}
\usepackage{latexsym}
\usepackage{epsfig}
  \usepackage{fullpage}

\newcommand\les{\lesssim}

\newcommand\bj{\textbf{j}}

\newcommand\R{\mathbb{R}}

\newcommand\C{\mathbb{C}}
\newcommand\Z{\mathbb{Z}}
\newcommand\N{\mathbb{N}}

\newcommand{\calC}{\mathcal C}
\newcommand{\calE}{\mathcal E}
\newcommand{\calH}{\mathcal H}

\newcommand{\ls}{{\lesssim}}

\newcommand\la{\langle}
\newcommand\ra{\rangle}
\newtheorem{theo}{Theorem}
\numberwithin{theo}{section} 
\newtheorem{lema}[theo]{Lemma}

\newtheorem{defin}[theo]{Definition}

\newtheorem{con}{Condition}
\newtheorem{cj}{Conjecture}

\numberwithin{equation}{section}

\begin{document}
\title{The multilinear restriction estimate: a short proof and a refinement }

\author[I. Bejenaru]{Ioan Bejenaru} \address{Department
  of Mathematics, University of California, San Diego, La Jolla, CA
  92093-0112 USA} \email{ibejenaru@math.ucsd.edu}

\begin{abstract} We provide an alternative and self contained proof of the main result of Bennett, Carbery, Tao in \cite{BeCaTa}
regarding the multilinear restriction estimate. The approach is inspired by the recent result of Guth \cite{Gu-easy} about the Kakeya version of multilinear restriction estimate. At lower levels of multilinearity we provide a refined estimate in the context of small support for one of the terms involved. 
\end{abstract}


\maketitle

\section{Introduction}

In \cite{BeCaTa} Bennett, Carbery and Tao established almost optimal multilinear restriction estimates. 
In this paper we provide an alternative proof of their main result and establish a refined version in the context of lower levels of multilinearity.
For a more in-depth introduction to the subject we refer the interested reader to \cite{BeCaTa}. 

For $n \geq 1$, let $U \subset \R^{n}$ be an open, bounded neighborhood of the origin and let $\Sigma: U \rightarrow \R^{n+1}$
be a smooth parametrization of a $n$-dimensional submanifold of $\R^{n+1}$. To this we associate the operator $\calE$ defined 
by
\[
\calE f(x) = \int_U e^{i x \cdot \Sigma(\xi)} f(\xi) d\xi. 
\]
with apriori domain $L^1(U)$. A fundamental question in Harmonic Analysis is the full range of $p,q$ for which 
$\calE : L^p(U) \rightarrow L^q(\R^{n+1})$ holds true. The original formulation of this question is in terms of the adjoint of $\calE$
and is known as the Restriction Conjecture, see \cite{BeCaTa} for more details. 

Multilinear versions of the restriction estimates have emerged in literature for various reasons. We start with the
$n+1$-multilinear restriction estimate. For  $1 \leq i \leq n+1$, let $\Sigma_i:U_i \rightarrow \R^{n+1}$ be smooth parametrizations as above,
satisfying
\begin{equation} \label{smooth}
\| \partial^\alpha \Sigma_i \|_{L^\infty(U_i)} \les_\alpha 1. 
\end{equation} 
and let $\calE_j$ be their associated operators. For $\zeta_i \in \Sigma_i(U_i)$, let $N_i(\zeta_i)$ be the unit normal at the surface $\Sigma_i(U_i)$;
the orientation is unimportant, hence it does not matter which unit normal is chosen. We assume the following transversality condition: there exists $\nu > 0$
such that
\begin{equation} \label{normal}
| det(N_1(\zeta_1), .., N_{n+1}(\zeta_{n+1}))| \geq \nu
\end{equation}
for all choices $\zeta_i \in \Sigma_i(U_i)$. The multilinear restriction conjecture is the following
\begin{cj} \label{MRCJ}
Suppose that \eqref{smooth} and \eqref{normal} hold true. Then the following holds true
\begin{equation} \label{MRE}
\| \Pi_{i=1}^{n+1} \calE_i f_i \|_{L^\frac{2}n(\R^{n+1})} \leq C \Pi_{i=1}^{n+1} \| f_i \|_{L^2(U_i)}. 
\end{equation}
where the constant $C$ depends on finitely many derivatives of $\Sigma_i$, $U_i$ ($i=1,..,n+1$)  and $\nu, n$.
\end{cj}
The regularity assumptions in \cite{BeCaTa} are of type $C^2$,
but we are not interested in optimizing this aspect in the current paper. 

The main result in \cite{BeCaTa} is a near-optimal version of the above conjecture:
\begin{theo}[Theorem 1.16, \cite{BeCaTa}] \label{MBCT}
Under the assumptions in Conjecture \ref{MRCJ}, for any $\epsilon > 0$, there is $C(\epsilon)$ such that 
the following holds true
\begin{equation} \label{AMRE}
\| \Pi_{i=1}^{n+1} \calE_i f_i \|_{L^\frac{2}n(B(0,R))} \leq C(\epsilon) R^\epsilon \Pi_{i=1}^{n+1} \| f_i \|_{L^2(U_i)}, 
\quad \forall f_i \in L^2(U_i), i=1,..,n+1,
\end{equation}
where $B(0,R) \subset \R^{n+1}$ is the ball of radius $R$ centered at the origin. 
\end{theo}
The Conjecture \ref{MRCJ} corresponds to obtaining \eqref{AMRE} with $\epsilon=0$ and it is currently an open problem. 
There is a multilinear Kakeya version of both \eqref{MRE} and \eqref{AMRE} which are slightly weaker statements than the corresponding
multilinear restriction ones. In \cite{BeCaTa} the authors prove the multilinear Kakeya version of \eqref{AMRE} and then obtain \eqref{AMRE}
from it by using a different argument. In a striking result, the multilinear Kakeya version of Conjecture \ref{MRCJ} 
was established by Guth in \cite{Gu-main} using tools from algebraic topology. 

In \cite{BeCaTa} the authors obtain a similar result for lower levels of multilinearity. One considers a similar setup with $k$
surface where $2 \leq k< n+1$. The assumption \eqref{smooth} is replaced by 
\begin{equation} \label{normal2}
vol (N_1(\zeta_1), .., N_{k}(\zeta_{k})) \geq \nu.
\end{equation}
for all choices $\zeta_i \in \Sigma_i(U_i)$. Here by $vol (N_1(\zeta_1), .., N_{k}(\zeta_{k}))$ we mean the volume of the $k$-dimensional parallelepiped spanned
by the vectors $N_1(\zeta_1), .., N_{k}(\zeta_{k})$. 
\begin{theo}[Section 5, \cite{BeCaTa}] \label{TL}
Assume $\Sigma_i, i=1,..,k$ satisfy \eqref{smooth} and \eqref{normal2}. Then for any $\epsilon > 0$, there is $C(\epsilon)$ such that 
the following holds true
\begin{equation} \label{LAMRE}
\| \Pi_{i=1}^{k} \calE_i f_i \|_{L^\frac{2}{k-1}(B(0,R))} \leq C(\epsilon) R^\epsilon \Pi_{i=1}^{k} \| f_i \|_{L^2(U_i)},
\quad \forall f_i \in L^2(U_i), i=1,..,k.
\end{equation}
\end{theo}

In dispersive PDE's whenever the iteration involves the use of the so-called bilinear $L^2$ type estimate, the bilinear version of \eqref{LAMRE} occurs somewhere in the argument, though without the $\epsilon$ loss. We can refer the interested reader to \cite{BeHeHoTa,BeHe-Z}, but note that this is such a widely used method, that it is virtually impossible to list all meaningful references. In most cases the bilinear $L^2$ type estimate comes with additional localization properties and this motivates the following refinement of the above result. We assume that $\Sigma_1$ has small support in some directions and ask how this affects the result above.

\begin{con} \label{C}
Assume that $\Sigma_1 \subset B(\calH,\mu)$, where $B(\calH,\mu)$ is the neighborhood of size $\mu$ of the $k$-dimensional affine subspace $\calH$. In addition assume that if $N_{i}, i=k+1,..,n+1$ is a basis of the normal space $\calH^\perp$ to $\calH$, then $N_1(\zeta_1),.., N_k(\zeta_k), N_{k+1},..,N_{n+1}$ are transversal in the sense \eqref{normal} for any choice $\zeta_i \in \Sigma_i$.  
 \end{con}
With this additional assumption on $\Sigma_1$, we obtain the following refinement of Theorem \ref{TL}:
\begin{theo} \label{MB}
Assume $\Sigma_i, i=1,..,k$ satisfy \eqref{smooth} and \eqref{normal2}. In addition, assume that $\Sigma_1$ satisfies Condition \ref{C}.
Then for any $\epsilon > 0$, there is $C(\epsilon)$ such that 
the following holds true
\begin{equation} \label{Lf}
\| \Pi_{i=1}^{k} \calE_i f_i \|_{L^\frac{2}{k-1}(B(0,R))} \leq C(\epsilon) \mu^{\frac{n+1-k}2} R^\epsilon \Pi_{i=1}^{k} \| f_i \|_{L^2(U_i)},
\quad \forall f_i \in L^2(U_i), i=1,..,k.
\end{equation}
\end{theo}

The bilinear versions (i.e. $k=2$) of Theorems \ref{TL} and \ref{MB} can be obtained without the $\epsilon$ loss, see  \cite{BeHeHoTa,BeHe-Z} for instance. What is special about the bilinear 
versions of \eqref{LAMRE} and  \eqref{Lf} is that it involves an $L^2$ type estimate, therefore it is equivalent to estimating 
a convolution of type $g_1 d\sigma_1 \ast g_2 d\sigma_2$ in $L^2$, where $d \sigma_1, d\sigma_2$ are measures supported on the hypersurfaces $\Sigma_1(U), \Sigma_2(U)$ respectively. Then obtaining \eqref{LAMRE} and the refined version in Theorem \ref{MB} is an easier task; moreover this gives directly the results without the $\epsilon$ loss. However, this approach relies on the use of Plancherel's theorem and, when $k \geq 3$, $L^\frac{2}{k-1} \ne L^2$. It is precisely this aspect that makes the multilinear estimate with $k \geq 3$ much harder than the bilinear estimate.  

The main goal of this paper is to provide a new proof of the multilinear restriction estimate in Theorem \ref{MBCT} and unveil the refined result at lower levels of multilinearity in Theorem \ref{MB}\footnote{While we do not prove Theorem \ref{TL} directly, its proof is a simplified version of the one we provide for Theorem \ref{MB}.}. The current arguments for Theorem \ref{MBCT} in \cite{BeCaTa} and \cite{Gu-easy} establish its Kakeya analogue and then appeal to a standard machinery described in  \cite{BeCaTa} to obtain Theorem \ref{MBCT}. The argument in \cite{BeCaTa} for proving the Kakeya analogue of Theorem \ref{MBCT} uses a continuous version of the standard induction on scale and it is rather involving. In a recent paper \cite{Gu-easy}, Guth provides an easier and more concise argument for the multilinear Kakeya version of \eqref{AMRE}. While the proof in \cite{Gu-easy} is short and elegant, one still needs to go back to \cite{BeCaTa} for an argument on how the near-optimal multilinear Kakeya estimate implies the near-optimal multilinear restriction estimate \eqref{AMRE}.

Inspired by the work in \cite{Gu-easy}, we are providing another short argument for Theorem \ref{MBCT}. The proof is provided directly for \eqref{AMRE}, not its Kakeya version, therefore there is no need to appeal to additional results. In this sense the proof is self-contained and this is one of the reasons to provide this new proof. The other reason for this new approach is that the method we develop provides an easy way to obtain the refined result in Theorem \ref{MB}.

Although inspired by the work in \cite{Gu-easy}, our geometric setup is closer in spirit to arguments used in previous work of the author, Herr and Tataru in \cite{BeHeTa}, which were later used by Bennett and Bez in \cite{BeBe}. In \cite{BeHeTa} a weaker version of \eqref{MRE} for $n=2$ was established: instead of the $L^1$ estimate for $\calE_1 f_1 \cdot \calE_2 f_2 \cdot \calE_3 f_3$, the $L^\infty$ estimate for its Fourier transform, $\widehat{\calE_1 f_1}  \ast \widehat{\calE_2 f_2} \ast \widehat{\calE_3 f_3}$, is provided.

There are two main ideas in our arguments: the use of a phase-space approach (localizations both on the physical and frequency side) 
to sort the geometry at the larger scale and the use of the discrete Loomis-Whitney inequality to pass from smaller scales to larger scales in the induction process.

Before ending the introduction, we highlight the following nonlinear character of the multilinear restriction estimate. A variant of the classical Loomis-Whitney inequality is the the following estimate
\begin{equation}  \label{LW}
\| \Pi_{i=1}^{n+1} \widehat{f_i d \calH_i} \|_{L^{\frac2{n}}(\R^{n+1})} \les \Pi_{i=1}^{n+1} \| f_i \|_{L^2(\calH_i)}, 
\end{equation}
where $\calH_i, i=1,..,n+1$ are transversal hyperplanes and $d \calH_i$ is the standard $n$-dimensional Lebesgue measure supported on $\calH_i$. By transversality we mean that if $N_i$ are (constant) unit normals to $\calH_i, i=1,..,n+1$,
then they satisfy \eqref{normal}. The proof \eqref{LW} is elementary. 
The multilinear restriction estimate is a non-linear generalization of the Loomis-Whitney inequality in the following sense: 
the hyperplanes $\calH_i$ are replaced by more general hypersurfaces $\Sigma_i$.  While the proof if \eqref{LW} is elementary, once the surfaces are allowed to have some curvature, things become far more complicated. 

\subsection*{Acknowledgement}
Part of this work was supported by a grant from the Simons Foundation ($\# 359929$, Ioan Bejenaru).
Part of this work was supported by the National Science Foundation under grant No. DMS-$1440140$ while the author was in residence at the 
Mathematical Research Sciences Institute in Berkeley, California, during the Fall 2015 semester. 

\section{Notation and discrete Loomis-Whitney inequalities} \label{CT}
 
\subsection{Notation} \label{NOT}

We use the standard notation $A \ls B$, meaning $A \leq C B$ for some universal $C$ which is independent of variables used in this paper, particularly it will be independent of $\delta$ and $R$ that appear in the main proof. By $A \ls_N B$ we mean $A \leq C(N) B$ and indicate that $C$ depends on $N$.  
 
We will work with $L^p(S), S \subset \R^n$ and, for that reason, we recall the standard estimate for superpositions of functions in $L^p$ for $0 < p \leq 1$:
\begin{equation} \label{triangle}
\| \sum_\alpha f_\alpha \|_{L^p}^p \leq \sum_\alpha \| f_\alpha \|_{L^p}^p. 
\end{equation}
 
We continue with the setup specific to our problem. Assume $\calH_1 \subset \R^{n+1}$ is a hyperplane (in the $\xi$ space) passing through the origin with normal $N_1$.
To keep notation compact, we will also denote by $\calH_1 \subset \R^{n+1}$ the hyperplane in the $x$ space passing through the origin with normal $N_1$. 
 We denote by $\mathcal{F}_1: \calH_1 \rightarrow \calH_1$ the standard Fourier transform, $x \rightarrow \xi$, and by $\mathcal{F}_1^{-1}$ the inverse Fourier transform, 
 $\xi \rightarrow x$. We denote the variables in $\R^{n+1}$ by $x=(x_1,x')$ respectively $\xi=(\xi_1,\xi')$, where $x_1,\xi_1$ are the coordinates along $N_1$ and $x',\xi'$ 
 are the coordinates along $\calH_1$. Obviously, $\mathcal{F}_1, \mathcal{F}_1^{-1}$ act on the variables $x',\xi'$ respectively.  We let $\pi_{N_1}: \R^{n+1} \rightarrow \calH_1$ the associated projection (in the $x$ space) along the normal $N_1$.
 
Assume $U_1  \subset \calH_1$ is open and  bounded. For $f: U_1 \rightarrow \C$, $f \in L^2(U_1)$ we define the operator 
$\calE_1: L^2(U_1) \rightarrow L^\infty(\R^{n+1})$ by 
\begin{equation} \label{E1}
\calE_1 f (x)= \int_{U_1} e^{i (x' \xi' + x_1 \varphi_1(\xi'))} f(\xi') d\xi'.
\end{equation}
 We highlight a commutator estimate which is needed due to the uncertainty principle. 
 It has a PDE flavor in it, but it can be stated in more classical fashion by studying the operator in \eqref{E1} from the perspective of oscillatory integrals.
  We define the differential operator $\nabla \varphi_1(\frac{D'}i)$ to be the operator with symbol $\nabla \varphi_1(\xi')$. 
 For any fixed $x_0 \in \R^{n+1}$, it holds true that
\begin{equation} \label{com}
(x'-x'_0-x_1 \nabla \varphi_1(\frac{D'}i))^N \calE_1 f =  \calE_1 (\mathcal{F}_1( (x'-x'_0)^N  \mathcal{F}_1^{-1} f)), \quad \forall N \in \N. 
\end{equation}
This is a direct computation using \eqref{E1} and it suffices to check it for $N=1$. 
The role of \eqref{com} will be to quantify localization properties of $\mathcal{F}_1^{-1} f$ on the hyperplanes $x_1=constant$.
Morally, \eqref{com} implies the following: if $\mathcal{F}_1^{-1} f$ (corresponding to $x_1=0$ in $\calE_1 f$) is concentrated in the set $|x'-x'_0| \les A$, then for fixed $x_1$, $\calE_1 f$ is concentrated in the set $|x'-x'_0-x_1 \nabla \varphi_1(\xi')| \les A$ where $\xi'$ covers the support of $f$. 
  
Next we prepare some geometric elements that are needed in the proof. Given $N_i, i=1,..,n+1$ transversal unit vectors in $\R^{n+1}$,
let $\calH_i \subset \R^{n+1}$ be the hyperplanes passing through the origin to which $N_i$ are normals. For each $i=1,..,n+1$, we define
$\mathcal{F}_i : \calH_i \rightarrow \calH_i$ the Fourier transform on $\calH_i$ and $\pi_{N_i}: \R^{n+1} \rightarrow \calH_i$ the projection onto $\calH_i$ as above. The vectors $N_i, i=1,..,n+1$ form a basis and 
the coordinates of a point $x \in \R^{n+1}$ are taken with respect to this basis. We construct $\mathcal{L}:=\{ z_1 N_1 + ... +z_{n+1} N_{n+1}:(z_1,..,z_{n+1}) \in \Z^{n+1} \}$ to be the oblique lattice in $\R^{n+1}$ generated by the unit vectors $N_1,..,N_{n+1}$. In each $\calH_i$ we construct the induced lattice $\mathcal{L}(\calH_i)=\pi_{N_i} (\mathcal{L})$; this is a lattice
since the projection is taken along a direction of the original lattice $\mathcal{L}$.

Given $r > 0$ we define $\calC(r)$ be the set of of parallelepipeds of size $r$ in $\R^{n+1}$ relative to the lattice $\mathcal{L}$; a parallelepiped in $\calC(r)$ has the following form $q(\bj): = [r (j_1-\frac12), r(j_1+\frac12)] \times .. \times [r (j_{n+1}-\frac12), r(j_{n+1}+\frac12)]$
where $\bj=(j_1,..,j_{n+1}) \in \Z^{n+1}$. For such a parallelepiped we define $c(q)=r \bj=(r j_1,.., r j_{n+1}) \in r\mathcal{L}$ to be its center.
Then, for each $i=1,..,n+1$, we let $\calC\calH_i(r) =\pi_{N_i} \calC(r)$ be the set of parallelepipeds of size $r$ in the hyperplane $\calH_i$.
Finally, given two parallelepipeds $q,q' \in \calC(r)$ or $\calC \calH_i(r)$ we define
$d(q,q')$ to be the distance between them when considered as subsets of the underlying space, let it be $\R^{n+1}$ or $\calH_i$.  

Let $\chi_0^n:\R^n \rightarrow [0,+\infty)$ be a Schwartz function, normalized in $L^1$, that is $\| \chi_0^n \|_{L^1}=1$,
and with Fourier transform supported on the unit ball. We fix $i \in\{ 1,..,n+1 \}$, $r> 0$ and define  $\mathcal{T}_i : \calH_i \rightarrow \calH_i$ to be the linear operator that takes $\mathcal{L}(\calH_i)$ to the standard lattice $\Z^{n}$ in $\calH_i$. Then for each $q \in \calC \calH_i(r)$, define
$\chi_q: \calH_i \rightarrow \R$ by
\[
\chi_{q}(x) = \chi_0^n (\mathcal{T}_i(\frac{x-c(q)}r))
\]
Notice that $\mathcal{F}_i \chi_{q}$ has Fourier support in the ball of radius $\les r^{-1}$. By the Poisson summation formula and properties of $\chi_0^n$, 
\begin{equation} \label{pois}
\sum_{q \in \calC \calH_i(r)} \chi_{q}=1.
\end{equation}
Using the properties of $\chi_q$, a direct exercise shows that for each $N \in \N$, the following holds true
\begin{equation} \label{SN}
\sum_{q \in \calC \calH_i(r)}  \| \la \frac{x-c(q)}r \ra^{N} \chi_{q} g \|_{L^2}^2 \les_N \| g \|_{L^2}^2
\end{equation}
for any $g \in L^2(\calH_i)$. Here, the variable $x$ is the argument of $g$ and belongs to $\calH_i$.

\subsection{Discrete versions of the Loomis-Whitney inequality} \label{DLW}
We end this section with two simple discrete versions of the continuous Loomis-Whitney inequality. The first one is the discrete version of \eqref{LW}; in the language introduced earlier, the following holds true
\begin{equation} \label{LWd}
\| \Pi_{i=1}^{n+1} g_i(\pi_{N_i}(z)) \|_{l^{\frac2{n}}(\mathcal{L})} \ls \Pi_{i=1}^{n+1} \| g_i \|_{l^2(\mathcal{L}(\calH_i))}. 
\end{equation}
where we assume that $N_i,i=1,..,n+1$ are transversal in the sense \eqref{normal}.

Next we provide a refinement of \eqref{LWd}. Given $k \in \N$ with $2 \leq k \leq n$, let $\calH_i \subset \R^{n+1}, i=1,..,k$ be $n$-dimensional hyperplanes passing through the origin and $N_i$ are their corresponding normals. We let $\calH \subset \calH_1$ be a subspace of dimension $k-1$ and let $N_{k+1},..,N_{n+1}$ be such that
$N_1, N_{k+1},..,N_{n+1}$ is an orthonormal basis to $\calH^\perp$, the normal space to $\calH$. We assume that $N_i, i=1,..,n+1$ are transversal in the sense \ref{normal} and note that this is invariant with respect to the choice of vectors $N_{k+1},..,N_{n+1}$. For $i=k+1,..,n+1$ we let $\calH_i$ be the hyperplanes passing through the origin with normal $N_i$.

Then as before we let $\pi_{N_i}, i=i,..,n+1$ be the corresponding projectors onto $\calH_i$. We define $\pi=\pi_{N_1} \circ \pi_{N_{k+1}} \circ .. \circ \pi_{N_{n+1}}$ to be the projector onto $\calH$. Then we let $\mathcal{L}$ be the lattice in $\R^{n+1}$ generated by $N_i$ and denote by $\mathcal{L}(\calH_i) = \pi_{N_i}(\mathcal{L}), i=2,..,k$ the induced lattice in $\calH_i$, while $\mathcal{L}(\calH) = \pi(\mathcal{L})$, the induced lattice in $\calH$. With this notation in place we have the following result:

\begin{lema} \label{Le2}
Assume $g_1 \in l^2(\mathcal{L}(\calH))$ and $g_i \in l^2(\mathcal{L}(\calH_i)), i=2,..,k$. Then the following holds true
\begin{equation} \label{LWd2}
\| g_1(\pi(z)) \Pi_{i=2}^{k} g_i(\pi_{N_i}(z)) \|_{l^{\frac2{k-1}}(\mathcal{L})} \ls \|g_1\|_{l^2(\mathcal{L}(\calH))}  \Pi_{i=2}^{k} \| g_i \|_{l^2(\mathcal{L}(\calH_i))}. 
\end{equation}
\end{lema}

\begin{proof} For $z \in \mathcal{L}$ we write $z=(z',z'')$ where $z'=(z_1,..,z_{k})$ collects the coordinates in the directions of $N_1,..,N_k$ and $z''$ collects the coordinates in the directions of $N_{k+1},..,N_{n+1}$. We fix $z''$, let $\mathcal{L}' \times \{z'' \}$ be the  sub-lattice of $\mathcal{L}$ obtained by fixing $z''$  and apply \eqref{LWd} to obtain
\[
\| g_1(\pi(\cdot,z'')) \Pi_{i=1}^{k} g_i(\pi_{N_i}(\cdot, z'')) \|_{l^{\frac2{k-1}}(\mathcal{L}' \times \{z''\})} \ls \|g_1(\pi_{N_1}(\cdot))\|_{l^2}  \Pi_{i=2}^{k} \| g_i(\pi_i(\cdot,z'')) \|_{l^2}.
\]
The first terms is motivated by the fact that $g_1(\pi_{N_1}(\cdot,z''))=g_1(\pi(\cdot))$. Then notice that, on the right-hand side above, inside the product $\Pi_{i=2}^{k}$, we have $k-1$ functions in $l^2$ with respect to the variable $z''$, thus leading to the desired $l^\frac{2}{k-1}$
estimate with respect to that variable and for the product. 

\end{proof}

\section{The induction argument for Theorem \ref{MBCT}} \label{SI}

Given some $0 < \delta \ll 1$ we split each domain $U_i$ into smaller pieces of diameter $\leq \delta$.
This, in turn, splits the surfaces $\Sigma_i(U_i)$ in the corresponding pieces. 
 It suffices to prove the multilinear estimate for each $\Sigma_i(U_i)$ being replaced
by one of its pieces, since then we can sum up the estimates for all possibles combinations of pieces using \eqref{triangle} and generate the original estimate at a cost of picking a factor of 
$\approx \left( (\delta^{-n})^{n+1} \right)^\frac{k-1}2=\delta^\frac{-kn(n+1)}2$. In the end of the argument, $\delta$ will be chosen in terms of absolute constants and $\epsilon$, but not $R$, and the factor $\delta^\frac{-kn(n+1)}2$  will be absorbed into $C(\epsilon)$. 

Now suppose that each $\Sigma_i(U_i)$ is as above, that is the diameter of $U_i$ is $ \leq \delta$. We choose and fix some $\zeta_i^0 \in \Sigma_i$, let $N_i=N_i(\zeta_i^0)$ be the 
normal to $\Sigma_i$ and let $\calH_i$ be the transversal hyperplane passing through the origin with normal $N_i(\zeta_i^0)$. Using a smooth change of coordinates, we can assume that $U_i \subset B_i(0,\delta) \subset \calH_i$ (where $B_i(0;\delta)$ is the ball in the hyperplane $\calH_i$ centered at the origin and of diameter $\delta$) and that 
\begin{equation} \label{E}
\calE_i f_i = \int_{U_i} e^{i (x' \xi' + x_i \varphi_i(\xi'))} f_i(\xi') d\xi',
\end{equation}
where $x=(x_i,x')$, $x_i$ is the coordinate in the direction of $N_i$ and $x'$ are the coordinates in the directions from $\calH_i$. 
Since the diameter of $U_i$ is $\les \delta$, 
it follows that $|\nabla \varphi_i(x)-\nabla \varphi_i(y)| \les \delta$ for any $x,y \in U_i$.  The rest of the argument will be provided for this setup. 

Using the normals $N_i$ we construct all entities described in Section \ref{NOT}. 

The proof of \eqref{AMRE} relies on estimating $\Pi_{i=1}^{n+1} \calE_i f_i$ on parallelepipeds on the physical side and analyze how the estimate
behaves as the size of the cube goes to infinity by using an inductive type argument with respect to the size of the parallelepiped.
As we move from one spatial scale to a larger one, we will have to tolerate slightly larger Fourier support in the argument. But this accumulation is in the form of a convergent geometric series, therefore the only harm it does is imposing an additional technical layer in the argument. This comes in the form of the margin concept previously used in the bilinear restriction theory, see  \cite{Tao-BW,Wo, Be1}.
For a function $f: \calH_i \rightarrow \C$ we define the margin 
\begin{equation} \label{defmrg}
\mbox{margin}^i(f) := \mbox{dist}(\mbox{supp} ( f), B_i(0;2\delta)^c), \quad i=1,..,n+1, 
\end{equation}
where $\mbox{supp} $ is the support of $f$.

\begin{defin} Given $R \geq \delta^{-2}$ we define $A(R)$ to be the best constant for which the estimate
\begin{equation}
\| \Pi_{i=1}^{n+1} \calE_i f_i \|_{L^{\frac2n}(Q)} \leq A(R) \Pi_{i=1}^{n+1} \| f_i \|_{L^2} 
\end{equation}
holds true for all parallelepipeds $Q \in \calC(R)$, with $f_i$
obeying the margin requirement
\begin{equation} \label{mrpg}
margin^i(f_i) \geq \delta-R^{-\frac12}.
\end{equation}
\end{defin}
The induction starts from $R \geq \delta^{-2}$ in order to be able to propagate the margin requirements.  

We provide an estimate inside any cube $Q \in \calC(\delta^{-1}R)$ based on prior information
on estimates inside cubes $q \in \calC(R) \cap Q$. Without restricting the generality of the argument, we assume
that $Q$ is centered at the origin and recall that each $q \in \calC(R) \cap Q$ has its center in
$R \mathcal{L}$. When such a  $q$ is projected using $\pi_{N_i}$ onto $\calH_i$ one obtains $\pi_{N_i} q \in \calC \calH_i(R)$.

Each $q \in \calC(R) \cap Q$ has size $R$ and the induction hypothesis is the following:
\begin{equation} \label{IO}
\| \Pi_{i=1}^{n+1} \calE_i f_i \|_{L^{\frac2{n}}(q)} \leq A(R) \Pi_{i=1}^{n+1} \| f_i \|_{L^2}. 
\end{equation}
We strengthen this to
\begin{equation} \label{INS}
\| \Pi_{i=1}^{n+1} \calE_i f_i \|_{L^{\frac2{n}}(q)} 
\les  A(R) \Pi_{i=1}^{n+1} \left( \sum_{q' \in \calC \calH_i(R)} \la \frac{d(\pi_{N_i} q,q')}R \ra^{-(2N-n^2)} \| \la \frac{x-c(q')}R \ra^{N} \chi_{q'} \mathcal{F}_i^{-1} f_i \|_{L^2}^2 \right)^\frac12
\end{equation}

The basic idea in \eqref{INS} is the following: if $q' \ne \pi_{N_i} q$, then 
$\calE_1 \mathcal{F}_1 ( \chi_{q'}  \mathcal{F}_1^{-1} f_1)$ has off-diagonal type contribution outside $q' \times [-\delta^{-1} R, \delta^{-1} R]$ (the interval stands for the $i$'th slot), thus it has off-diagonal
type contribution to the left-hand side of \eqref{INS}. This is achieved as follows: fix $i=1$ and $q' \in \calC \calH_1(R)$. With $x=(x_1,x')$ we have
\[
\begin{split}
& \| (x'- c(q')-x_1 \nabla \varphi_1(\xi'_0)) \calE_1 \mathcal{F}_1 ( \chi_{q'}  \mathcal{F}_1^{-1} f_1) \cdot \Pi_{i=2}^{n+1}  \calE_i f_i \|_{L^{\frac2{n}}(q)} \\
= & \| (x'- c(q')-x_1 \nabla \varphi_1(\xi')) \calE_1 \mathcal{F}_1 ( \chi_{q'} \mathcal{F}_1^{-1} f_1) \cdot \Pi_{i=2}^{n+1}  \calE_i f_i  \|_{L^{\frac2{n}}(q)} \\
+ & \| x_1 ( \nabla \varphi_1(\xi'_0) - \nabla \varphi_1(\xi')) \calE_1 \mathcal{F}_1 ( \chi_{q'} \mathcal{F}_1^{-1} f_1) \cdot \Pi_{i=2}^{n+1}  \calE_i f_i  \|_{L^{\frac2{n}}(q)} \\
= & \|  \calE_1 \mathcal{F}_1 ( (x'-c(q'))  \chi_{q'} \mathcal{F}_1^{-1} f_1) \cdot \Pi_{i=2}^{n+1}  \calE_i f_i  \|_{L^{\frac2{n}}(q)} \\
+ & \| x_1  \calE_1 \mathcal{F}_1 ( ( \nabla \varphi_1(\xi'_0) - \nabla \varphi_1(\xi')) \chi_{q'} \mathcal{F}_1^{-1} f_1) \cdot \Pi_{i=2}^{n+1}  \calE_i f_i  \|_{L^{\frac2{n}}(q)} \\
\leq & A(R) \left(  \| (x'-c(q')) \chi_{q'} \mathcal{F}_1^{-1} f_1 \|_{L^2} + \delta^{-1} R \| (\nabla \varphi_1(\xi'_0) - \nabla \varphi_1(\xi')) \chi_{q'} \mathcal{F}_1^{-1} f_1 \|_{L^2} \right) \Pi_{i=2}^{n+1} \| f_i \|_{L^2}  \\
\les & A(R) \left(  \| (x'-c(q')) \chi_{q'} \mathcal{F}_1^{-1} f_1 \|_{L^2} + R \| \chi_{q'} \mathcal{F}_1^{-1} f_1 \|_{L^2} \right)  \Pi_{i=2}^{n+1} \| f_i \|_{L^2} \\
\les & R A(R) \| \la \frac{x'-c(q')}{R} \ra \chi_{q'} \mathcal{F}_1^{-1} f_1 \|_{L^2}  \Pi_{i=2}^{n+1} \| f_i \|_{L^2}
\end{split}
\]
We have used the following: \eqref{com} in justifying the equality between the terms on the second and fourth line, the induction hypothesis and the fact that inside $Q$ we have $|x_1| \les \delta^{-1}R$ to justify the inequality in the sixth line. Note that it is in the above use of the induction estimate for $\calE_1 \mathcal{F}_1 ( (x'-c(q'))  \chi_{q'} \mathcal{F}_1^{-1} f_1)$ that we need to tolerate
the relaxed support of $f_1$. The margin of $f_1$ is $\geq \delta-(\delta^{-1} R)^{-\frac12}=\delta-\delta^\frac12 R^{-\frac12}$ is affected by the convolution
$\mathcal{F}_1 ((x'-c(q'))  \chi_{q'})$ by a factor of at most $C R^{-1}$ which is smaller than $\frac12 \delta^\frac12 R^{-\frac12}$, provided that $\delta$ is small relative to $C^{-1}$. Hence the new margin is $\geq \delta-\frac12 \delta^\frac12 R^{-\frac12} \geq \delta - R^{-\frac12}$, this being the required margin for using the induction hypothesis on cubes of size $R$. 

For any $q \in \calC(R) \cap Q$ and $x' \in \pi_{N_1}(q)$, it holds that 
$\la \frac{x'- c(q')-x_1 \nabla \varphi_1(\xi'_0)}R \ra \approx \la \frac{d(\pi_{N_1}(q),q')}R \ra$. 
This is justified by the fact that $|x_1| \les \delta^{-1}R$ and $|\nabla \varphi_1(\xi'_0)| \leq \delta$, 
therefore the contribution of $|x_1 \nabla \varphi_1(\xi'_0)| \leq R$ is negligible. From this and the previous set of estimates,
we conclude that
\[
 \| \calE_1 \mathcal{F}_1 ( \chi_{q'} \mathcal{F}_1^{-1} f_1) \cdot \Pi_{i=2}^{n+1}  \calE_i f_i \|_{L^{\frac2{n}}(q)} \les 
  \la  \frac{d(\pi_{N_1}q,q')}{R} \ra^{-1} \| \la \frac{x'-c(q')}{R} \ra \chi_{q'} \mathcal{F}_1^{-1} f_1 \|_{L^2}  \Pi_{i=2}^{n+1} \| f_i \|_{L^2}
\]
Repeating the argument gives
\[
 \| \calE_1 \mathcal{F}_1 ( \chi_{q'} \mathcal{F}_1^{-1} f_1) \cdot \Pi_{i=2}^{n+1}  \calE_i f_i \|_{L^{\frac2{n}}(q)} \les_N
  \la  \frac{d(\pi_{N_1}q,q')}{R} \ra^{-N} \| \la \frac{x'-c(q')}{R} \ra^N \chi_{q'} \mathcal{F}_1^{-1} f_1 \|_{L^2}  \Pi_{i=2}^{n+1} \| f_i \|_{L^2}
\]

Using \eqref{pois}, \eqref{triangle} and the above, we obtain
\[
\begin{split}
& \| \calE_1 f_1 \cdot \Pi_{i=2}^{n+1}  \calE_i f_i \|_{L^{\frac2{n}}(q)}^\frac2n \\
 \leq  &
\sum_{q' \in \calC \calH_1(R)} \| \calE_1 \mathcal{F}_1 ( \chi_{q'} \mathcal{F}_1^{-1} f_1) \cdot \Pi_{i=2}^{n+1}  \calE_i f_i \|_{L^{\frac2{n}}(q)}^\frac2n \\
 \les_N & \left( \sum_{q' \in \calC \calH_{1}(R)} \la  \frac{d(\pi_{N_1}q,q')}{R} \ra^{-N \cdot \frac2n} \| \la \frac{x'-c(q')}{R} \ra^N \chi_{q'} \mathcal{F}_1^{-1} f_1 \|_{L^2}^\frac2n \right)  \Pi_{i=2}^{n+1} \| f_i \|^\frac2n_{L^2} \\
 \les_N & \left( \sum_{q' \in \calC \calH_{1}(R)} \la  \frac{d(\pi_{N_1}q,q')}{R} \ra^{-(2N-n^2)} \| \la \frac{x'-c(q')}{R} \ra^N \chi_{q'} \mathcal{F}_1^{-1} f_1 \|_{L^2}^2\right)^\frac1n  \Pi_{i=2}^{n+1} \| f_i \|^\frac2n_{L^2}.
\end{split}
\]
In justifying the last inequality, we have used the simple estimate for sequences 
\[
\| a_i \cdot b_i \|_{l^{\frac2{n}}_i} \les \| a_i \|_{l^2_i}  \| b_i \|_{l^{\frac2{n-1}}_i} 
\]
together with the straightforward estimate
\[
\| \la  \frac{d(\pi_{N_1}q,q')}{R} \ra^{-\frac{n^2}2} \|_{l^\frac{2}{n-1}_{q'}} \les 1. 
\]
Note that the previous inequality is \eqref{INS} with the improvement for $f_1$.
By repeating the procedure for all other terms $f_2,..,f_{n+1}$ to conclude with \eqref{INS}. 

Using \eqref{INS} we are ready to conclude the argument by invoking the discrete Loomis-Whitney inequality in \eqref{LWd}. 
We define the functions $g_i: \mathcal{L}(\calH_i) \rightarrow \R$ by
\[
g_i(\bj)= \left( \sum_{q' \in \calC \calH_{i}(R)} \la  \frac{d(q(\bj),q')}{R} \ra^{-(N-2n^2)} \| \la \frac{x'-c(q')}{R} \ra^N \chi_{q'} \mathcal{F}_i^{-1} f_i \|^2_{L^2} \right)^\frac{1}2, \bj \in \mathcal{L}(\calH_i) .
\]
From \eqref{SN}, it is easy to see that for $N$ large enough (depending only on $n$),
$g_i \in l^2(\Z^n)$ with
\[
\| g_i \|_{l^2(\mathcal{L}(\calH_i))} \ls \| f_i \|_{L^2}.
\] 
Using \eqref{LWd} we conclude that \eqref{INS} implies 
\[
\| \Pi_{i=1}^{n+1} \calE_i f_i \|_{L^{\frac2{n}}(Q)}  \ls A(R) \Pi_{i=1}^{n+1} \| f_i \|_{L^2}.
\]
Thus we obtain
\[
A(\delta^{-1} R) \leq C A(R)
\]
for a constant $C$ that is independent of $\delta$ and $R$. Iterating this gives $A(\delta^{-N} r) \leq C^{N} A(r)$.
Therefore $\max_{r \in [0,\delta^{-2}]} A(\delta^{-N} r) \leq C^N \max_{r \in [0,\delta^{-2}]} A(r)= C^N C(\delta)$. This is simply
obtained from the uniform pointwise bound
\begin{equation} \label{final}
\| \Pi_{i=1}^{n+1} \calE_i f_i \|_{L^\infty} \ls \Pi_{i=1}^{n+1}  \| \calE_i f_i \|_{L^\infty} \ls \Pi_{i=1}^{n+1}  \| f_i \|_{L^2}
\end{equation}
which is then integrated over arbitrary cubes of size $\leq \delta^{-2}$.

For  $R \in [\delta^{-N},\delta^{-N-1}]$, the above implies
\[
A(R) \leq C^N C(\delta) \leq R^\epsilon C(\delta)
\]
provided that $C^N \leq \delta^{-N\epsilon}$. Therefore choosing $\delta=C^{-\frac{1}{\epsilon}}$ leads to the desired result. 

\section{The induction argument for Theorem \ref{MB}} \label{SI2}

The proof follows the same steps as in the previous Section with some modifications.
 Note that \eqref{Lf} says something meaningful over \eqref{LAMRE} only if, in the language used above, $\mu \ll \delta$, or else the gain of $\mu^{\frac{n+1-k}2}$ is undistinguishable from $C(\epsilon)$ that is translated into $\tilde C(\delta)$.

For each $i=1,..,k$ we fix $\zeta_i^0 \in \Sigma_i(U_i)$, $N_i=N_i(\zeta_i^0)$ and let $\calH_i$ be the hyperplane on the physical side passing through the origin with normal $N_i$. We denote by $\pi_{N_i}$ the projection onto $\calH_i$ along $N_i$. Then, we choose a basis $N_i, i=k+1,..,n+1$ of the normal plane to $\calH$, let $\calH_i$ be the hyperplane on the physical side passing through the origin with normal $N_i$ and denote by $\pi_{N_i}$ the projection onto $\calH_i$ along $N_i$. 
 The set $\{ N_i \}_{i=1,..,n+1}$ is a basis of $\R^{n+1}$ and throughout this section the coordinates of a point 
are written in this basis.

Then as we described in Section \ref{NOT}, we construct the lattice $\mathcal{L}$, the set of parallelepipeds $\calC(r)$, the induced lattices 
$\calC \calH_i(r)$ in $\calH_i$ and the induced set of parallelepipeds $\calC \calH_i(r)$.

Next, a key point in the argument is that in the induction argument the localization at scale $\mu$ of $f_1$ is conserved exactly in all directions 
from $\calH^\perp$ and not through some margin process as in the proof of Theorem \ref{MBCT}. We now make this precise.

We work under the hypothesis that $U_i \subset B_i(0,\delta), i=2,..,k$, where $B_i(0,\delta)$ is the ball in the hyperplane $\calH_i$. For a function 
$f_i: \calH_i \rightarrow \C$ its margin is defined as before, see \eqref{defmrg}. 

We work under the hypothesis that $U_1 \subset B'(0,\delta) \times B''(0,\mu) \subset \calH_1$, where $B'(0,\delta)$ is the ball in the hyperplane $\calH_1 \cap \calH$ centered at the origin and of diameter $\delta$ and $B'(0,\mu)$ is the ball in the hyperplane $(\calH_1 \cap \calH)^\perp$ centered at the origin and of diameter $\mu$. For a function 
$f: \calH_1 \rightarrow \R$ its margin is define by
\[
\mbox{margin}^1(f) := \inf_{\xi''} \mbox{dist}(\mbox{supp}_{\xi'} (f(\cdot,\xi'')), B'(0,2\delta)^c), 
\]
where $\mbox{supp}_{\xi'}$ is the support of $f$ in the $\xi'$ variable. With these notations in place, we define

\begin{defin} Given $R \geq \delta^{-2}$ we define $A(R)$ to be the best constant for which the estimate
\begin{equation}
\| \Pi_{i=1}^{k} \calE_i f_i \|_{L^{\frac2{k-1}}(Q)} \leq A(R) \Pi_{i=1}^{k} \| f_i \|_{L^2} 
\end{equation}
holds true for all cubes $Q \in \calC(R)$, with $f_i$
obeying the margin requirement
\begin{equation} \label{mrpg2}
margin^i(f_i) \geq \delta-R^{-\frac12}
\end{equation}
and $f_1$ is supported in a neighborhood of size $\mu$ of $\calH \cap \calH_1 \subset \calH_1$.  
\end{defin}

We start with the cube $Q$ of size $\delta^{-1} R$ centered at the origin. For each $q \in \calC(R) \cap Q$, the induction hypothesis is the following:
\begin{equation} \label{IO3}
\| \Pi_{i=1}^{k} \calE_i f_i \|_{L^{\frac2{k-1}}(q)} \leq A(R) \Pi_{i=1}^{k} \| f_i \|_{L^2}. 
\end{equation}
As we did before, we will strengthen it, keeping in mind that we do not want to alter the support of $f_1$ in directions from $(\calH_1 \cap \calH)^\perp$. To do so we need a little more notation that goes along the lines of Section \ref{DLW}. Let $\pi=\pi_{N_1} \circ \pi_{N_{k+1}} \circ .. \circ \pi_{N_{n+1}}$.
We consider the subspace $\calH_1 \cap \calH$ of dimension $k-1$ and construct $\calC (\calH_1 \cap \calH)(r)=\pi \calC(r)$ be the set of parallelepipeds in 
$\calH_1 \cap \calH$ obtained by projecting parallelepipeds from $\calC(r)$. Their centers belong to the lattice $\mathcal{L} (\calH_1 \cap \calH)=\pi (\mathcal{L})$.
Based on this, we define $ \mathfrak{S}_1(r)$ to be the set of infinite parallelepipedical strips
$\mathfrak{s}=q \times (\calH_1 \cap \calH)^\perp \subset \calH_1$, where $q \in \calC (\calH_1 \cap \calH)(r)$. We denote
by $c(\mathfrak{s}):=c(q) \subset \mathcal{L}(\calH_1 \cap \calH)$ be the center of the strip.  We note that given $q_1, q_2 \in \calC \calH_1(r)$, then 
$\pi q_1, \pi q_2 \subset \calH_1$ belong to the same parallelepipedical strip in 
$\mathfrak{S}_1(r)$ if and only if $\pi q_1=\pi q_2$. For $q \in \calC_1(r)$, we let $\mathfrak{s} (\pi q)$ be the infinite parallelepipedical strip it belongs to as a subset in $\mathfrak{S}_1(r)$. Finally, given a strip $\mathfrak{s} \in \mathfrak{S}_1(r)$ we define $\chi_{\mathfrak{s} }: \calH_1 \rightarrow \R$
\[
\chi_{\mathfrak{s}}(x) = \chi_0^{k-1} (\mathcal{T}(\frac{\pi(x)-c(\mathfrak{s})}{r}))
\]
where $\chi_0^{k-1}: \R^{k-1} \rightarrow \R$ is entirely similar to the $\chi_0^n$ introduced in Section \ref{NOT}, expect that it acts on $\R^{k-1}$ instead of $\R^{n}$ and  $\mathcal{T} : \calH \cap \calH_1 \rightarrow \calH \cap \calH_1$
is the linear operator taking $\mathcal{L}(\calH \cap \calH_1)$ to the standard lattice $\Z^{k-1}$ in $\calH \cap \calH_1$. 
A key property of $\chi_{\mathfrak{s}}$ is that it does not depend on the variables $x_{k+1},..,x_{n+1}$, its coordinates in the subspace $\calH^\perp$.

Now we claim the following strengthening of \eqref{IO3}:
\begin{equation} \label{INS3}
\begin{split}
& \| \Pi_{i=1}^{k} \calE_i f_i \|_{L^{\frac2{k-1}}(q)} \\
\les_N &   A(R) \Pi_{i=2}^{k} 
\left( \sum_{q' \in \calC\calH_i(R)} \la \frac{d(\pi_{N_i} q,q')}R \ra^{-(2N-n^2) } 
\| \la \frac{x-c(q')}R \ra^{N} \chi_{q'} \mathcal{F}_i^{-1} f_i \|_{L^2}^2 \right)^\frac{1}2 \\
& \cdot
\left( \sum_{\mathfrak{s}' \in \mathfrak{S}_1(R)} \la \frac{d(\mathfrak{s} (\pi q),\mathfrak{s}')}R \ra^{-(2N-n^2) } \| \la \frac{\pi (x)-c(\mathfrak{s}')}R \ra^{N} \chi_{\mathfrak{s}'} \mathcal{F}_1^{-1} f_1 \|_{L^2}^2 \right)^\frac{1}2
\end{split}
\end{equation}
Here is very important in the second term above is that $\chi_{\mathfrak{s}'}$ does not depend on the $x_{k+1},..,x_{n+1}$ variables, thus
it does not affect the support of $f_1$ in the directions $\xi_{k+1},..,\xi_{n+1}$. As a consequence the margin requirements are propagated 
as required by the new definition. The argument for obtaining \eqref{INS3} is entirely similar to the one used to derive
\eqref{INS} with the only difference being that for the $\calE_1 f_1$ we use the multiplier 
\[
(x_2,..,x_{k})-c(\mathfrak{s}') - x_1 \nabla_{\xi_2,..,\xi_k} \varphi_1(\xi'_0)
\]
which is consistent with the fact that we do not want to alter the variables $x_{k+1},..,x_{n+1}$, so as to keep the support properties 
of $f_1$ intact in the directions of  $\xi_{k+1},..,\xi_{n+1}$.

As before, we define the functions $g_i:  \mathcal{L} (\calH_i) \rightarrow \R$ for $i=2,..,k$ by the same formula
\[
g_i(\bj)= \left( \sum_{q' \in \calC\calH_{i}(R)} \la  \frac{d(q(\bj),q')}{R} \ra^{-(2N-n^2) } \| \la \frac{x'-c(q')}{R} \ra^N \chi_{q'} \mathcal{F}_i^{-1} f_i \|_{L^2}^2 \right)^\frac12, 
\]
for $\bj \in \mathcal{L}(\calH_i)$, while $g_1: \mathcal{L}(\calH \cap \calH_1) \rightarrow \R$ by
\[
g_1(\bj)= \left( \sum_{\mathfrak{s}' \in \mathfrak{S}_1(R)} \la  \frac{d(\mathfrak{s} (\pi  q(\bj)),\mathfrak{s}')}{R} \ra^{-(2N-n^2) } \| \la \frac{\pi(x)-c(\mathfrak{s}')}{R} \ra^N \chi_{\mathfrak{s}'} \mathcal{F}_1^{-1} f_1 \|_{L^2}^2 \right)^\frac12, 
\]
for $\bj \in \mathcal{L}(\calH \cap \calH_1)$.Then as before we have
\[
\| g_i \|_{l^2(\mathcal{L}(\calH_i)} \ls \| f_i \|_{L^2}, \quad i=2,..,k
\] 
while
\[
\| g_1 \|_{l^2(\mathcal{L}(\calH \cap \calH_1))} \ls \| f_1 \|_{L^2}.
\]
Then we apply \eqref{LWd2} to conclude with
\[
\| \Pi_{i=1}^{k} \calE_i f_i \|_{L^{\frac2{k-1}}(Q)}  \ls A(R) \Pi_{i=1}^{k} \| f_i \|_{L^2}.
\]
From this point on we continue as in the previous argument. It is in the derivation of \eqref{final} that we pick the gain in $\mu$ from the support of $f_1$ in the directions 
from $\calH \cap \calH_1$ (which is not changed through the induction process) as follows:
\[
\| \calE_1 f_1 \|_{L^\infty} \ls \mu^{\frac{n+1-k}2} \| f_1 \|_{L^2}
\]
This finishes the proof.

\bibliographystyle{amsplain} \bibliography{HA-refs}

\end{document}